\documentclass[11pt]{article}
\usepackage{smile}

\usepackage{fullpage}
\usepackage{lscape}
\usepackage{bigints}
\usepackage{framed}
\usepackage{mdframed}
\usepackage{enumerate}
\usepackage[T1]{fontenc}
\usepackage{moresize}
\usepackage{bm}
\usepackage{bbm}
\usepackage{dsfont}
\usepackage{amsmath}
\usepackage{amssymb}
\usepackage{amsthm}
\usepackage{amsfonts}
\usepackage{stmaryrd}
\usepackage{array}
\usepackage{mathrsfs}
\usepackage{mathtools} 
\usepackage{extarrows}
\usepackage{stackrel}
\usepackage{relsize,exscale}
\usepackage{scalerel}
\usepackage[nodisplayskipstretch]{setspace}
\usepackage{color}
\usepackage[usenames,dvipsnames]{xcolor}
\usepackage{cancel}
\usepackage{soul}
\usepackage{undertilde}
\usepackage{xfrac}
\usepackage{siunitx}
\usepackage{graphicx}
\usepackage{float}
\usepackage{rotating}
\usepackage{subcaption}
\usepackage{overpic}
\usepackage[all]{xy}
\usepackage{tikz}
\usetikzlibrary{arrows,matrix,positioning,calc,automata,patterns}
\usepackage{booktabs}
\usepackage{dcolumn}
\usepackage{multirow}
\usepackage{diagbox}
\usepackage{tabularx}
\usepackage{verbatim}
\usepackage{listings}
\usepackage[ruled,vlined]{algorithm2e}
\usepackage{fancyvrb}
\usepackage{hyperref}
\usepackage[round]{natbib}
\usepackage{sectsty}

\hypersetup{
    bookmarks=true,         
    unicode=false,          
    pdftoolbar=true,        
    pdfmenubar=true,        
    pdffitwindow=false,     
    pdfstartview={FitH},    
    pdftitle={My title},    
    pdfauthor={Author},     
    pdfsubject={Subject},   
    pdfcreator={Creator},   
    pdfproducer={Producer}, 
    pdfkeywords={key1, key2}, 
    pdfnewwindow=true,      
    colorlinks=true,        
    linkcolor=blue,         
    citecolor=blue,         
    filecolor=blue,         
    urlcolor=cyan           
}

\usepackage{stackengine}
\stackMath
\newcommand\tenq[2][1]{%
\def\useanchorwidth{T}%
\ifnum#1>1%
\stackunder[0pt]{\tenq[\numexpr#1-1\relax]{#2}}{\!\scriptscriptstyle\thicksim}%
\else%
\stackunder[1pt]{#2}{\!\scriptstyle\thicksim}%
\fi%
}

\makeatletter
\DeclareRobustCommand\widecheck[1]{{\mathpalette\@widecheck{#1}}}
\def\@widecheck#1#2{%
    \setbox\z@\hbox{\m@th$#1#2$}%
    \setbox\tw@\hbox{\m@th$#1%
       \widehat{%
          \vrule\@width\z@\@height\ht\z@
          \vrule\@height\z@\@width\wd\z@}$}%
    \dp\tw@-\ht\z@
    \@tempdima\ht\z@ \advance\@tempdima2\ht\tw@ \divide\@tempdima\thr@@
    \setbox\tw@\hbox{%
       \raise\@tempdima\hbox{\scalebox{1}[-1]{\lower\@tempdima\box
\tw@}}}%
    {\ooalign{\box\tw@ \cr \box\z@}}}
\makeatother

\def\tr{\mathop{\text{tr}}\kern.2ex}

\def\P{{\mathrm P}}

\def\E{{\mathrm E}}

\renewcommand{\Pr}{\mathrm{P}}

\newcommand{\op}{\mathrm{op}}

\newcommand{\indep}{\perp \!\!\!\perp}

\newcolumntype{L}[1]{>{\raggedright\let\newline\\\arraybackslash\hspace{0pt}}m{#1}}
\newcolumntype{C}[1]{>{  \centering\let\newline\\\arraybackslash\hspace{0pt}}m{#1}}
\newcolumntype{R}[1]{>{ \raggedleft\let\newline\\\arraybackslash\hspace{0pt}}m{#1}}
\newcolumntype{d}[1]{D{.}{.}{#1}}
\newcolumntype{H}{>{\setbox0=\hbox\bgroup}c<{\egroup}@{}}
\newcolumntype{Z}{>{\setbox0=\hbox\bgroup}c<{\egroup}@{\hspace*{-\tabcolsep}}}
\newcolumntype{b}{X}
\newcolumntype{s}{>{\hsize=.5\hsize}X}

\def\PP{{\mathbb{P}}}

\numberwithin{equation}{section}

\newtheorem{theorem}{Theorem}[section]

\newtheorem{proposition}{Proposition}[section]
\newtheorem{assumption}{Assumption}[section]
\newtheorem{corollary}{Corollary}[section]
\newtheorem{innercustomgeneric}{\customgenericname}
\providecommand{\customgenericname}{}
\newcommand{\newcustomtheorem}[2]{%
  \newenvironment{#1}[1]
  {%
   \renewcommand\customgenericname{#2}%
   \renewcommand\theinnercustomgeneric{##1}%
   \innercustomgeneric
  }
  {\endinnercustomgeneric}
}
\newcustomtheorem{customdefinition}{Definition}
\newcustomtheorem{customdefinitions}{Definitions}
\newcustomtheorem{customtheorem}{Theorem}
\newcustomtheorem{customassumption}{Assumption}
\newcustomtheorem{customlemma}{Lemma}
\newcustomtheorem{customexample}{Example}
\theoremstyle{definition}

\newtheorem{remark}{Remark}[section]

\makeatletter
\newcommand{\mylabel}[2]{#2\def\@currentlabel{#2}\label{#1}}
\makeatother

\usepackage{enumitem}

\graphicspath{{./fig3/}}

\allowdisplaybreaks

\begin{document}

\setlength{\abovedisplayskip}{5pt}
\setlength{\belowdisplayskip}{5pt}
\setlength{\abovedisplayshortskip}{5pt}
\setlength{\belowdisplayshortskip}{5pt}
\hypersetup{
  colorlinks,
  breaklinks,
  urlcolor=blue,
  linkcolor=blue,
  pdftitle={Imputation is all you need},
  pdfauthor={Fang Han and Peng Ding}
}


\title{\LARGE Imputation is all you need:\\
double robustness, semiparametric efficiency, and automatic covariate balance for estimating the average treatment effect}

\author{
Fang Han\thanks{Department of Statistics, University of Washington, Seattle, WA 98195, USA; e-mail: {\tt fanghan@uw.edu}}
~~~{\rm and}~~
Peng Ding\thanks{Department of Statistics, University of California, Berkeley, CA 94720, USA; email: {\tt pengdingpku@berkeley.edu}}\thanks{The authors thank Avi Feller for bringing \cite{imbens2005mean} to their attention.}}

\date{\today}

\maketitle

\vspace{-1em}

\begin{abstract}
Imputation-based causal estimation is typically viewed as relying exclusively on an outcome model, in contrast to augmented inverse-probability weighting, whose consistency is protected by fitting two nuisance models.  This paper argues that this view can be misleading by highlighting a hidden dual weighting structure in least-squares sieve regression imputation. Although only outcome regressions are explicitly fitted, the resulting imputation estimator admits an exact weighting representation whose induced weights balance every function in the sieve space and the corresponding population weighting functions are the $L^2$ projections of the inverse propensity scores onto the same sieve space.  This projection structure yields an implicit form of double robustness and, under standard sieve approximation and growth conditions, asymptotic linearity with the efficient influence function.  Thus, weighting, covariate balance, double robustness, and semiparametric efficiency can all emerge from imputation alone through the geometry of least-squares projection.
\end{abstract}


\section{Introduction}

Imputation provides a natural way to conceptualize causal estimation under the potential-outcomes framework.  With a binary treatment, only one of the two potential outcomes is observed for each unit.  A natural strategy is therefore to predict the missing potential outcome using units receiving the opposite treatment and then proceed as if the resulting completed data were fully observed.  This perspective has a long history in the literature of missing data and causal inference  \citep{rubin1978bayesian, rubin2004multiple,imbens2015causal}, and imputation-based methods continue to play an important role in empirical practice; see, for example, \cite{imbens2024causal}.

Several early papers developed the theoretical foundations of the imputation approach in important settings.  \cite{heckman1998matching} studied nonparametric regression-based imputation and developed its large-sample theory.  \cite{hahn1998role}, in examining the role of the propensity score in efficient estimation, further showed that efficient estimators of average treatment effects can be constructed by averaging data completed through nonparametric imputation.  \cite{imbens2005mean} studied least-squares sieve regression imputation and established its semiparametric efficiency.  \cite{abadie2006large} and \cite{lin2021estimation} subsequently developed large-sample theories for nearest-neighbor matching estimators, clarifying both their interpretation as imputation procedures and the conditions under which they fail to, or can, attain semiparametric efficiency.

Almost in parallel, \citet{robins1994estimation} and \citet{robins1995semiparametric}
developed inverse-probability weighting (IPW) estimators and their efficient augmentations. 
These estimators involve two nuisance components: a model for the treatment-assignment mechanism and a model for the conditional outcome distribution.  The resulting augmented inverse-probability weighting (AIPW) estimators possess an important robustness property \citep{scharfstein1999adjusting,bang2005doubly}: they can remain consistent when either the propensity-score model or the outcome-regression model is correctly specified, without requiring both to be correct.  This property is now known as \emph{double robustness}.  In modern causal inference, the augmented representation is also closely connected to the efficient influence function (EIF) and its associated orthogonality with respect to nuisance perturbations \citep{tsiatis2006semiparametric}; related principles underlie regression-adjusted imputation \citep{abadie2011bias,lin2022regression} and, more recently, double/debiased machine learning \citep{chernozhukov2018double}.

This history has led to a familiar conceptual distinction.  Pure imputation is typically viewed as an outcome-regression procedure whose validity depends on the quality of the model used to predict the missing potential outcomes, whereas doubly robust estimators combine an outcome regression with a model for the treatment-assignment mechanism \citep{bang2005doubly,cao2009improving}.  The latter remain consistent if either the outcome-regression model or the propensity-score model is correctly specified, without requiring both to be correct \citep{bang2005doubly}.  From this conventional perspective, imputation alone would not appear to enjoy the same protection unless propensity-score information or an explicit augmentation term is introduced; see, for example, \cite{seaman2018introduction} for a concise account of this conventional view.

We argue that this distinction is misleading.  Consider fitting the outcome regression separately within the treated and control samples and imputing the missing potential outcomes using the resulting sieve regressions.  We show that, even though the propensity score is never explicitly estimated, the second nuisance component is already embedded in the procedure.  More precisely, we establish that
\begin{enumerate}[itemsep=-.5ex,label=(\roman*)]
\item the averages of the imputed outcomes admit exact finite-sample weighting representations in the sense of  \citet{chattopadhyay2021implied};
\item the induced weights exactly balance every function in the sieve space in the sense of \citet[Equation (4)]{chan2016globally};
\item at the population level, the corresponding weighting functions are the weighted $L^2$ projections of the inverse propensity scores onto the same sieve space in the sense of \citet[Equation (4.9)]{imbens2005mean} and \citet[Proof of Theorem 2]{wang2023distributed}.
\end{enumerate}
Thus, the inverse-propensity component that appears explicitly in AIPW arises automatically from the least-squares normal equations.

This set of observations immediately yields a double-robustness property, echoing the influential discussion in \citet[Section~3]{robins2007comment}.  The population bias of the imputation estimator admits an exact representation involving the product of two approximation residuals: one for the outcome regression and the other for the corresponding inverse propensity score, or equivalently the propensity odds.  Consequently, the bias is bounded by the product of the two associated weighted $L^2$ approximation errors.  The familiar double-robustness structure therefore emerges from imputation \emph{alone}, without explicitly estimating a propensity score or introducing an augmentation term.

The same mechanism also explains semiparametric efficiency.  Under standard sieve approximation and growth conditions, the implicit weighting functions converge to the inverse propensity scores, while the fitted regressions converge to the corresponding outcome regressions.  The resulting imputation estimator is asymptotically linear with influence function equal to the EIF for the average treatment effect (ATE), and hence attains the semiparametric efficiency bound.  This complements the efficiency results of \cite{hahn1998role} and \cite{heckman1998matching}, and also echoes an observation of \cite{kang2007demystifying}, who, in their empirical comparison of doubly robust methods, remarked that ``[n]one of the DR methods we tried, however, improved upon the performance of simple regression-based prediction of the missing values.'' 
This paper provides an additional theoretical explanation, complementing the efficiency result in \cite{imbens2005mean} and the robustness properties discussed in \citet[Section~3]{robins2007comment}, for why a seemingly simple regression-imputation procedure can perform so well: in the least-squares sieve setting, it already contains the weighting and orthogonality structure typically associated with explicitly doubly robust procedures.

Lastly, we emphasize that this paper builds substantially on a reinterpretation of the results of \cite{wang2023distributed}; see, also, \citet[Equation (4.9)]{imbens2005mean}.
In the context of mean estimation with missing responses, \cite{wang2023distributed} study least-squares sieve imputation and establish its semiparametric efficiency.  More importantly for the present paper, the argument surrounding their proof of Theorem 2 exploits the least-squares normal equations to obtain the essential product-bias bound and explicitly observes that the approximation bias is small whenever either the outcome regression or the inverse response probability is well approximated by the sieve.  

Our motivation for revisiting this phenomenon is therefore not to rediscover the product-bias calculation itself \citep{wang2004semiparametric, rotnitzky2021characterization}, but to isolate and develop the structural mechanism behind it within the potential-outcomes framework that is now standard in modern causal inference and more directly connected to contemporary causal methodology.  In particular, we show that the same least-squares geometry generates exact finite-sample balancing weights, whose population counterparts are weighted $L^2$ projections of the inverse propensity scores, and that this primal-dual structure yields arm-specific double robustness and recovers the EIF for the ATE.  From this perspective, weighting, covariate balance, double robustness, and semiparametric efficiency emerge as different manifestations of the same underlying projection structure.

\subsection{Other related work}

Our results are connected to three additional strands of research.  The first concerns the \emph{weighting interpretation of linear regression}.  Ordinary least squares can itself be written as a weighted estimator; see, for example, \cite{angrist2009mostly} and \cite{imbens2015matching}.  More strikingly, \citet[Section~3]{robins2007comment} and \cite{kline2011oaxaca} showed, in different settings, that regression estimators can exhibit robustness properties more commonly associated with weighting methods: the former for estimation of a population mean with incomplete data, and the latter for estimation of treatment effects.  This connection has subsequently been developed more systematically by \cite{chattopadhyay2021implied}.  A common theme in this literature is that regression can implicitly generate balancing weights even when no weighting procedure is constructed explicitly.

The second strand concerns \emph{covariate-balancing weights}.  Rather than estimating a propensity score and subsequently inverting it, these methods construct weights directly to balance selected functions of the covariates \citep{hainmueller2012entropy,imai2014covariate,zubizarreta2015stable}.  Subsequent work showed that covariate-balancing methods can also achieve double robustness and semiparametric efficiency, properties more commonly associated with AIPW estimators \citep{zhao2017entropy,chan2016globally}.  More recently, \cite{bruns2026augmented} incorporated an explicit augmentation step into covariate-balancing estimators, paralleling the role of bias correction in imputation-based methods \citep{abadie2011bias,lin2022regression}. 

A third, more general perspective is provided by the recent literature on estimation of \emph{Riesz representers}.  For the ATE, the relevant arm-specific Riesz representers are precisely the inverse propensity scores.  Conventional AIPW approaches obtain these objects by estimating the propensity score and then inverting it, which can be numerically unstable, particularly when the estimated score is close to zero (cf. Basu's elephant due to \citet{basu2010essay}).  Rather than proceeding through this inversion step, \cite{chernozhukov2022debiased} developed regularized estimators of Riesz representers for global and local regression functionals, while \cite{chernozhukov2026adversarial} proposed an adversarial approach that allows the representer to be learned over rich function classes.  Similar in spirit to the covariate-balancing literature, these approaches avoid direct propensity-score inversion, treat the regression function and its Riesz representer as two explicit nuisance components, estimate them separately, and then combine them through an orthogonal correction.

The present paper connects these three perspectives.  In particular, we argue that the assignment mechanism  need not be treated as a separate nuisance component, because its relevant weighted $L^2$ projection onto the sieve space is already generated by the regression fit itself.  Accordingly, the implicit weights studied here may be viewed simultaneously as balancing weights and as automatically constructed sieve approximations to the corresponding Riesz representers.

\subsection{Paper organization}

The remainder of the paper proceeds as follows.  Section~\ref{sec:setup} introduces the causal model and the sieve-imputation estimator.  Section~\ref{sec:theory} develops the exact implicit-weighting and covariate-balance representations and establishes implicit double robustness and semiparametric efficiency under explicit sieve-approximation and growth conditions.  Section~\ref{sec:discussion} concludes with a broader discussion of the implications of these results for modern least-squares regression methods. Section~\ref{sec:proofs} contains the proofs. 

\section{Setup and sieve imputation}\label{sec:setup}

\subsection{Causal setup}

We adopt the standard potential-outcomes framework with a binary treatment under the following assumption. 

\begin{assumption}\label{ass:dgp}
Assume that $
\{X_i,W_i,Y_i(0),Y_i(1)\}_{i=1}^n
$
are independent copies of $(X,W,Y(0),Y(1))$, where $Y(0)$ and $Y(1)$ are the scalar-valued potential outcomes under control and treatment, respectively, $W\in\{0,1\}$ is the treatment indicator, and $X\in\mathcal X\subseteq\mathbb R^d$ is a vector of pretreatment covariates.  
\end{assumption}

The observed data are
\[
Z_i:=(X_i,W_i,Y_i),
\qquad
Y_i:=W_iY_i(1)+(1-W_i)Y_i(0),
\qquad i\in[n]:=\{1,\ldots,n\}.
\]
For $\omega\in\{0,1\}$, let
\[
\mu_\omega:=\E\{Y(\omega)\}.
\]
Our parameter of interest is the ATE
\[
\tau:=\E\{Y(1)-Y(0)\}
=\mu_1-\mu_0.
\]
We focus on the canonical setting of unconfounded observational studies under the following assumption. 

\begin{assumption}[Unconfoundedness and overlap]\label{ass:unconf}
We have 
\[
\{Y(0),Y(1)\}\indep W\mid X.
\]
Moreover, there exists a constant $c_e>0$ such that
\[
c_e\le e(X)\le1-c_e
\qquad\text{almost surely},
\]
where $e(x):=\Pr(W=1\mid X=x)$ is the propensity score. 
\end{assumption}

It is convenient to write
\[
\delta_\omega:=\mathbbm 1(W=\omega),
\qquad
\pi_1(x):=e(x),
\qquad
\pi_0(x):=1-e(x),
\]
so that $\E(\delta_\omega\mid X=x)=\pi_\omega(x)$.  We also define the corresponding \emph{propensity-odds functions}
\[
\rho_\omega(x)
:=
\frac{1-\pi_\omega(x)}{\pi_\omega(x)},
\qquad
\omega\in\{0,1\},
\]
so that
\[
\rho_1(x)=\frac{1-e(x)}{e(x)},
\qquad
\rho_0(x)=\frac{e(x)}{1-e(x)}.
\]
Define the outcome-regression functions
\begin{align}\label{eq:outcome-reg}
m_\omega(x):=\E\{Y(\omega)\mid X=x\},
\qquad
\omega\in\{0,1\}.
\end{align}

Let $\P$ denote the data-generating law and $\PP_n$ the corresponding empirical measure.  For a generic measurable function $f$, write
\[
\P f:=\E\{f(Z)\},
\qquad
\PP_n f:=\frac1n\sum_{i=1}^n f(Z_i).
\]

\subsection{Sieve regression-based imputation}

For each positive integer $K$, let
\[
\mathcal V_K
:=
\operatorname{span}\{v_{1K},\ldots,v_{KK}\}
\]
denote a $K$-dimensional \emph{sieve}, or \emph{feature}, space, where the basis functions are allowed to depend on $K$, and write
\[
V_K(x)
:=
\bigl(v_{1K}(x),\ldots,v_{KK}(x)\bigr)^\top.
\]
We assume throughout that the constant function belongs to every sieve space and, without loss of generality, take
\[
v_{1K}(x)\equiv1,
\qquad
K=1,2,\ldots.
\]
As $K$ increases, the sieve spaces are assumed to be sufficiently rich to approximate the relevant target functions, in the spirit of \cite{newey1997convergence} and \cite{chen2007large}.  In the language of modern machine learning, sieve regression may equivalently be viewed as constructing an increasingly rich \emph{feature map} $V_K(\cdot)$ and performing least-squares regression on the resulting features.

For $\omega\in\{0,1\}$, define the least-squares sieve estimator
\begin{equation}\label{eq:beta-hat}
\widehat\beta_\omega
:=
\arg\min_{\beta\in\mathbb R^K}
\PP_n\left[
\delta_\omega
\{Y-V_K(X)^\top\beta\}^2
\right],
\end{equation}
and let
\[
\widehat m_\omega(x)
:=
V_K(x)^\top\widehat\beta_\omega
\]
be the resulting predictor of $Y(\omega)$ at $X=x$.  For each unit $i\in[n]$, the missing potential outcomes are imputed as
\[
\widehat Y_i(1)
=
W_iY_i+(1-W_i)\widehat m_1(X_i),
\qquad
\widehat Y_i(0)
=
(1-W_i)Y_i+W_i\widehat m_0(X_i).
\]
The resulting sieve-imputation estimator of the ATE is
\begin{equation}\label{eq:tau-imp}
\widehat\tau_{\mathrm{imp}}
:=
\frac1n\sum_{i=1}^n
\{\widehat Y_i(1)-\widehat Y_i(0)\}.
\end{equation}
Equivalently,
\begin{align}\label{eq:mu-hat}
\widehat\tau_{\mathrm{imp}}
=
\widehat\mu_1-\widehat\mu_0,
\qquad {\rm with}~~
\widehat\mu_\omega
:=
\PP_n\left[
\delta_\omega Y+(1-\delta_\omega)\widehat m_\omega(X)
\right],
\end{align}
which coincides with the Imbens-Newey-Ridder imputation estimator \citep[Section~3.3]{imbens2005mean}.

For the population analysis, define
\begin{equation}\label{eq:beta-star}
\beta_{\omega,K}^*
:=
\arg\min_{\beta\in\mathbb R^K}
\E\left[
\delta_\omega
\{Y-V_K(X)^\top\beta\}^2
\right],
\qquad
\omega\in\{0,1\},
\end{equation}
and let
\[
m_{\omega,K}(x)
:=
V_K(x)^\top\beta_{\omega,K}^*.
\]
Under unconfoundedness (cf. Assumption \ref{ass:unconf} above), $m_{\omega,K}$ is hence equivalently the weighted $L^2$ projection of $m_\omega$ in \eqref{eq:outcome-reg} onto $\mathcal V_K$ under the inner product
\[
\langle f,g\rangle_\omega
:=
\E\{\delta_\omega f(X)g(X)\}.
\]
The corresponding population imputation functional is
\[
\tau_K
:=
\E\Big[
WY+(1-W)m_{1,K}(X)
-(1-W)Y-Wm_{0,K}(X)
\Big].
\]

We close this section by clarifying the connection between $\widehat\tau_{\mathrm{imp}}$ in \eqref{eq:tau-imp} and some earlier imputation and matching estimators.  \cite{heckman1998matching} studied matching as a nonparametric regression procedure, using kernel and local-polynomial methods to impute missing potential outcomes and developing the corresponding large-sample theory.  \cite{hahn1998role} showed that nonparametric imputation can attain the semiparametric efficiency bound for the ATE, although his construction estimates the outcome regressions through ratios involving a nonparametric estimate of the propensity score.  \cite{abadie2006large}, in turn, interpreted nearest-neighbor matching as an imputation procedure and showed that its nonsmooth and local nature generates a matching bias that can preclude $\sqrt n$-consistency, while matching with a fixed number of neighbors is generally not semiparametrically efficient.  Subsequent work by \cite{abadie2011bias}, \cite{lin2021estimation}, and \cite{lin2022regression} augments such local imputation procedures with additional bias-correction steps, thereby bringing them closer in structure to AIPW estimators.

In contrast, the imputation procedure proposed by \cite{imbens2005mean} and studied here fits the outcome regressions directly by least-squares sieve regression, without explicitly estimating the propensity score or constructing local matches.  In particular, unlike the procedures in \cite{abadie2011bias}, \cite{lin2021estimation}, and \cite{lin2022regression}, \emph{no} additional bias-correction step is required.  Explaining the mechanism underlying this phenomenon is the main focus of the next section.

\section{Main results}\label{sec:theory}

\subsection{The implicit weighting structure}\label{sec:implicit-weighting}

This section shows that least-squares sieve imputation admits an exact finite-sample weighting representation (Proposition~\ref{prop:implicit-weighting}), whose population counterpart satisfies the three-way balancing identity in \citet[Equation~(4)]{chan2016globally} (Proposition~\ref{prop:population-balance}) and is the weighted $L^2$ projection of the inverse propensity score onto the same sieve space (Proposition~\ref{prop:inverse-ps-projection}).  Notably, all of these properties emerge automatically as by-products of the sieve-imputation structure, rather than being imposed by design.

In detail, define
\[
\widehat Q_{\omega,K}
:=
\PP_n\{\delta_\omega V_K(X)V_K(X)^\top\},
\qquad
\overline V_K
:=
\PP_n V_K(X),
\]
and, whenever $\widehat Q_{\omega,K}$ is nonsingular, let
\[
\widehat w_{\omega,K}(x)
:=
V_K(x)^\top
\widehat Q_{\omega,K}^{-1}
\overline V_K
\]
denote the corresponding \emph{implicit weighting function}.  With the notation above,
\[
\widehat\beta_\omega
=
\widehat Q_{\omega,K}^{-1}
\PP_n\!\left\{
\delta_\omega V_K(X)Y
\right\},
\qquad
\widehat\mu_\omega
:=
\PP_n\!\left[
\delta_\omega Y
+
(1-\delta_\omega)V_K(X)^\top\widehat\beta_\omega
\right].
\]

Proposition~\ref{prop:implicit-weighting} below shows that sieve imputation is algebraically equivalent to a weighting estimator whose induced weights exactly balance every function in the sieve space, even though no propensity score is explicitly estimated.  This may therefore be viewed as an imputation counterpart to analogous balancing-weight arguments in, e.g., \citet[Proposition~3]{chattopadhyay2021implied} and \citet[Proposition~3.1]{bruns2026augmented}.

\begin{proposition}[Exact implicit weighting and balance]
\label{prop:implicit-weighting}
Suppose that $v_{1K}\equiv 1$ and $\widehat Q_{\omega,K}$ is nonsingular.  Then, for each $\omega\in\{0,1\}$, the imputation estimator $\widehat\mu_\omega$ in \eqref{eq:mu-hat} admits the weighting representation
\[
\widehat\mu_\omega
=
\PP_n\{\delta_\omega\widehat w_{\omega,K}(X)Y\}.
\]
Moreover, the implicit weights exactly balance every function in the sieve space:
\[
\PP_n\{\delta_\omega\widehat w_{\omega,K}(X)V_K(X)\}
=
\PP_n V_K(X).
\]
\end{proposition}

At the population level, define the analogues of $\widehat Q_{\omega,K}$, $\overline V_K$, and $\widehat w_{\omega,K}$ by
\begin{equation}\label{eq:Q-q-w}
Q_{\omega,K}
:=
\E\{\delta_\omega V_K(X)V_K(X)^\top\},
\qquad
q_K
:=
\E\{V_K(X)\},
\end{equation}
and
\begin{equation}\label{eq:w-pop}
w_{\omega,K}(x)
:=
V_K(x)^\top Q_{\omega,K}^{-1}q_K.
\end{equation}
Equation~\eqref{eq:population-balance-function} below then shows that $w_{1,K}$ and $w_{0,K}$ satisfy, on the sieve space $\mathcal V_K$, the population three-way balancing identity in \citet[Equation~(4)]{chan2016globally}.

\begin{proposition}
\label{prop:population-balance}
Suppose that $\E\|V_K(X)\|^2<\infty$, where $\|\cdot\|$ denotes the Euclidean norm, and $Q_{\omega,K}$ is nonsingular.  Then, for each $\omega\in\{0,1\}$ and every $a\in\mathcal V_K$,
\[
\E\!\left\{
\delta_\omega
w_{\omega,K}(X)
a(X)
\right\}
=
\E\{a(X)\}.
\]
Consequently, for every $a\in\mathcal V_K$,
\begin{equation}\label{eq:population-balance-function}
\E\!\left\{
Ww_{1,K}(X)a(X)
\right\}
=
\E\!\left\{
(1-W)w_{0,K}(X)a(X)
\right\}
=
\E\{a(X)\}.
\end{equation}
\end{proposition}

Proposition~\ref{prop:inverse-ps-projection} below further shows that the population weights $w_{\omega,K}$ underlying the implicit weighting and balancing structures in Propositions~\ref{prop:implicit-weighting} and~\ref{prop:population-balance} are the unique weighted $L^2$ projections of $1/\pi_\omega$ onto $\mathcal V_K$.  

\begin{proposition}
\label{prop:inverse-ps-projection}
Suppose that $Q_{\omega,K}$ is nonsingular and
\[
\E\{\pi_\omega(X)^{-1}\}<\infty.
\]
Then $w_{\omega,K}$ in \eqref{eq:w-pop} is the unique weighted $L^2$
projection of $1/\pi_\omega$ onto $\mathcal V_K$, where the weighted
inner product is
\[
\langle f,g\rangle_\omega
:=
\E\{\delta_\omega f(X)g(X)\}.
\]
Equivalently,
\begin{equation}\label{eq:inverse-ps-projection}
\E\left[
\delta_\omega V_K(X)
\left\{
\frac{1}{\pi_\omega(X)}
-
w_{\omega,K}(X)
\right\}
\right]
=0.
\end{equation}
Moreover, since $1\in\mathcal V_K$, $w_{\omega,K}-1$ is the unique
weighted $L^2$ projection of
$\rho_\omega=\pi_\omega^{-1}-1$ onto $\mathcal V_K$.
\end{proposition}

Proposition \ref{prop:inverse-ps-projection} follows from the argument surrounding the proof of Theorem 2 in \cite{wang2023distributed}. In light of the implicit covariate-balancing structure developed in Section~\ref{sec:implicit-weighting}, the proposition may also be viewed as an imputation analogue of related projection characterizations in balancing-weight arguments; see, e.g., \citet[Theorem~1(i)]{zhao2017entropy} and \citet[Theorem~1]{chattopadhyay2021implied}.

Proposition~\ref{prop:inverse-ps-projection} also parallels the key identity in \citet[Lemma~3.1]{lin2022regression}, building on the ideas in \cite{lin2021estimation}.  In their setting, regression-adjusted imputation based on linear smoothers admits an AIPW representation in which quantities induced by the smoothing weights play the role of the propensity-score, or equivalently density-ratio, component.  Thus, in both settings, an imputation procedure that explicitly models only the missing potential outcomes implicitly generates a weighting component ordinarily associated with propensity-score methods.  The underlying mechanisms are, however, different.  In \citet{lin2022regression}, the relevant weights arise from local or more general linear smoothers and enter through an AIPW representation, whereas here they arise directly from global least-squares sieve imputation, exactly balance the sieve space, and have population counterparts given by weighted $L^2$ projections of the inverse propensity scores.

Lastly, we note that the connections among regression, weighting, covariate balance, and robustness also have important precedents in the survey-sampling literature. For instance, \citet{deville1992calibration} showed that the generalized regression estimator admits an equivalent calibration-weighting interpretation: rather than motivating the estimator through a regression model, one may construct weights that remain close to the original design weights while satisfying calibration equations that reproduce known population totals of auxiliary variables.  This provides an early form of the regression-weighting duality underlying Proposition~\ref{prop:implicit-weighting}. A related robustness phenomenon appears in \citet{kott1994note} in the context of survey nonresponse.  Kott considers simultaneously a response model and a parametric model, and develops regression and imputation estimators that retain protection when one of these modeling frameworks is misspecified.  This may be viewed as an early precursor of double robustness.  In this sense, the results above connect the classical calibration perspective of \citet{deville1992calibration} with the two-model robustness perspective of \citet{kott1994note} through the geometry of least-squares projection.  The following subsections show how this structure leads to the product-bias and double-robustness results developed below.

\subsection{Regularity conditions and approximation errors}

We now turn to double robustness and semiparametric efficiency. The arguments below are adapted and simplified from
\citet[Theorem~2]{wang2023distributed}. We begin with the following conditions.

\begin{assumption}[Moments]\label{ass:moments}
There exists a constant $C_Y<\infty$ such that, for each
$\omega\in\{0,1\}$,
\[
\sup_{x\in\mathcal X}
\E\!\left[
\{Y(\omega)-m_\omega(X)\}^2
\mid X=x
\right]
\le C_Y
\]
and
\[
\E\{m_\omega(X)^2\}\le C_Y.
\]
\end{assumption}

\begin{assumption}[Sieve regularity]\label{ass:sieve}
Let $K=K_n$, possibly diverging with $n$, and define
\[
Q_K
:=
\E\{V_K(X)V_K(X)^\top\},
\qquad
\zeta_K
:=
\sup_{x\in\mathcal X}\|V_K(x)\|,
\]
The following conditions hold.
\begin{enumerate}
\item[(i)] The sieve contains the constant function, $1\in\mathcal V_K$, and, without loss of generality,
\[
v_{1K}(x)\equiv1,
\qquad
K=1,2,\ldots.
\]

\item[(ii)] There exist constants $0<c_Q<C_Q<\infty$, independent of $K$, such that
\[
c_Q
\le
\lambda_{\min}(Q_K)
\le
\lambda_{\max}(Q_K)
\le
C_Q.
\]

\item[(iii)] The sieve dimension and basis envelope satisfy
\[
\frac{K}{n}\to0,
\qquad
\frac{\zeta_K^2\log (2K)}{n}\to0.
\]
\end{enumerate}
\end{assumption}

\begin{remark}
Assumption \ref{ass:moments} corresponds to Assumption 3.1(iii) and (iv) in \citet{lin2022regression}. The conditions in Assumption~\ref{ass:sieve} are standard in the least-squares series literature; see, for example, \cite{newey1997convergence}, \cite{chen2007large}, and \cite{belloni2015some}.  Under standard normalizations, commonly used sieve bases have well-controlled envelopes.  For example, power-series bases may satisfy $\zeta_K=O(K)$, whereas Fourier series, splines, compactly supported wavelets, and piecewise-polynomial bases typically satisfy $\zeta_K=O(K^{1/2})$; see \cite{newey1997convergence} and \cite{chen2007large}.  The eigenvalue condition in Assumption~\ref{ass:sieve}(ii) is likewise standard; see, for example, \citet[Condition~A.2 and Proposition 2.1]{belloni2015some}.  
\end{remark}

For $\omega\in\{0,1\}$, recall the treatment-specific population Gram matrix
\[
Q_{\omega,K}
=
\E\{\delta_\omega V_K(X)V_K(X)^\top\}
=
\E\{\pi_\omega(X)V_K(X)V_K(X)^\top\}.
\]
Assumption~\ref{ass:unconf} implies
\[
c_eQ_K
\preceq
Q_{\omega,K}
\preceq
(1-c_e)Q_K,
\qquad
\omega\in\{0,1\},
\]
and hence Assumption~\ref{ass:sieve}(ii) yields
\[
c_ec_Q
\le
\lambda_{\min}(Q_{\omega,K})
\le
\lambda_{\max}(Q_{\omega,K})
\le
(1-c_e)C_Q.
\]

For $\omega\in\{0,1\}$, define the \emph{outcome-regression approximation error}
\begin{align}\label{eq:Delta-m}
\Delta_{m,\omega,K}
:=&
\inf_{a\in\mathcal V_K}
\left[
\E\left\{
\delta_\omega
\bigl(m_\omega(X)-a(X)\bigr)^2
\right\}
\right]^{1/2}\notag\\
=&
\left[
\E\left\{
\delta_\omega
\bigl(m_\omega(X)-m_{\omega,K}(X)\bigr)^2
\right\}
\right]^{1/2},
\end{align}
where the equality follows from the weighted $L^2$ projection characterization of $m_{\omega,K}$.  Analogously, define the \emph{propensity-odds approximation error}
\begin{align}\label{eq:Delta-rho-w}
\Delta_{\rho,\omega,K}
:=&
\inf_{a\in\mathcal V_K}
\left[
\E\left\{
\delta_\omega
\bigl(\rho_\omega(X)-a(X)\bigr)^2
\right\}
\right]^{1/2}\notag\\
=&
\left[
\E\left\{
\delta_\omega
\left(
\frac1{\pi_\omega(X)}
-
w_{\omega,K}(X)
\right)^2
\right\}
\right]^{1/2},
\end{align}
where the second equality follows from
Proposition~\ref{prop:inverse-ps-projection} and the fact that
$1\in\mathcal V_K$.

\subsection{Implicit double robustness}

\citet[Section~3]{robins2007comment} showed that linear regression estimators can exhibit a form of double robustness.  \citet[Section~4.1]{imbens2005mean} derived a closely related residual-product representation for the bias, while the proof of Theorem~2 of \citet{wang2023distributed} makes the corresponding product-of-approximation-errors structure explicit. We formulate this phenomenon here for estimation of the ATE and, through Propositions~\ref{prop:implicit-weighting}--\ref{prop:inverse-ps-projection}, make the role of the implicit weighting functions $w_{\omega,K}(\cdot)$ more explicit.

\begin{theorem}[Product-bias representation and bound]
\label{thm:implicit-dr}
Suppose Assumptions~\ref{ass:dgp}, \ref{ass:unconf}, \ref{ass:moments}, and \ref{ass:sieve}(i)-(ii) hold.  Then
\begin{equation}\label{eq:dr-bias-compact}
\tau_K-\tau
=
B_{1,K}-B_{0,K},
\end{equation}
where, for each $\omega\in\{0,1\}$,
\begin{align}
B_{\omega,K}
&:=
\E\{m_{\omega,K}(X)-m_\omega(X)\}
\notag\\
&=
\E\left[
\delta_\omega
\left\{
\frac{1}{\pi_\omega(X)}
-
w_{\omega,K}(X)
\right\}
\{m_{\omega,K}(X)-m_\omega(X)\}
\right].
\label{eq:dr-bias-arm}
\end{align}
Consequently,
\begin{equation}\label{eq:product-bound}
|\tau_K-\tau|
\le
\Delta_{\rho,1,K}\Delta_{m,1,K}
+
\Delta_{\rho,0,K}\Delta_{m,0,K},
\end{equation}
where the approximation errors $\{\Delta_{\rho,\omega,K},\Delta_{m,\omega,K}:\omega\in\{0,1\}\}$ were introduced in \eqref{eq:Delta-m} and \eqref{eq:Delta-rho-w}.
\end{theorem}

Theorem~\ref{thm:implicit-dr} shows that the approximation bias is second order in the sense that it is controlled by products of the outcome-regression and propensity-odds approximation errors.  This immediately yields the following double-robustness result for the sample estimator.

\begin{corollary}[Implicit double robustness]
\label{cor:implicit-dr}
Suppose Assumptions~\ref{ass:dgp}-\ref{ass:sieve} hold.  Then
\begin{equation}\label{eq:sample-pop}
\widehat\tau_{\mathrm{imp}}-\tau_K
=
o_\P(1).
\end{equation}
Hence, if $\Delta_{\rho,1,K}\Delta_{m,1,K} +
\Delta_{\rho,0,K}\Delta_{m,0,K}
=
o(1)$,
then
\[
\widehat\tau_{\mathrm{imp}}-\tau=o_\P(1).
\]
Moreover, for any fixed $K$, the corresponding population bias vanishes arm by arm whenever
\[
m_\omega\in\mathcal V_K
\qquad\text{or}\qquad
\rho_\omega\in\mathcal V_K,
\qquad \omega\in\{0,1\}.
\]
Consequently, $\tau_K=\tau$ if either condition holds for each treatment arm, and the two arms may rely on different sides of this condition.
\end{corollary}

\begin{remark}
    Our central theme might sound at odds with that in \citet{robins1997toward}. They studied a related survey sampling setting with known survey weights (that is, the inverse of propensity scores) but arbitrarily complicated outcome model. Under their model, the approximation error $m_{1,K}(X)-m_1(X)$ can be arbitrarily large, so they concluded that we must use estimators based on survey weights. Since the survey weights are essentially covariates, a canonical approach is to use them directly in the outcome model to ensure the approximation error $ {\pi_1(X)}^{-1} - w_{1,K}(X)$ is small. In particular, we can add the known $1/\pi_1(X)$ in the sieve to ensure that $ {\pi_1(X)}^{-1} - w_{1,K}(X) = 0$; see \citet{rubin1985use} and \citet{little2004robust} for related methods. 
\end{remark}

\subsection{Semiparametric efficiency}\label{sec:efficiency}

Double robustness concerns the bias of the population imputation functional.  Semiparametric efficiency requires, in addition, that the first-order stochastic component induced by the implicit weights converge to the EIF.  We therefore impose the following stronger conditions.

\begin{assumption}[Conditions for efficiency]\label{ass:efficiency}
The following conditions hold.
\begin{enumerate}
\item[(i)] The sieve complexity satisfies
\begin{equation}\label{eq:eff-complexity}
\frac{\zeta_K^2K\log (2K)}{n}\to0.
\end{equation}

\item[(ii)] For each $\omega\in\{0,1\}$,
\[
\Delta_{m,\omega,K}\to0,
\qquad
\Delta_{\rho,\omega,K}\to0,
\]
and
\begin{equation}\label{eq:eff-product-rate}
\sqrt n
\left\{
\Delta_{\rho,1,K}\Delta_{m,1,K}
+
\Delta_{\rho,0,K}\Delta_{m,0,K}
\right\}
\to0.
\end{equation}

\item[(iii)] The population implicit weights are uniformly bounded: there exists a constant $C_w<\infty$ such that
\[
\sup_{K\ge1}
\max_{\omega\in\{0,1\}}
\sup_{x\in\mathcal X}
|w_{\omega,K}(x)|
\le C_w.
\]
\end{enumerate}
\end{assumption}

Condition~\eqref{eq:eff-product-rate} is the familiar product-rate condition.  Neither nuisance component needs to be approximated at the parametric rate; rather, the product of their approximation errors must be $o(n^{-1/2})$.  Condition~\eqref{eq:eff-complexity} is a standard strengthening of the sieve-growth condition needed for root-$n$ inference.  For example, when $\zeta_K\lesssim\sqrt K$, it is implied by
\[
\frac{K^2\log (2K)}{n}\to0.
\]

Following \cite{hahn1998role}, define the EIF for the ATE as
\begin{equation}\label{eq:eif}
\psi(Z)
:=
m_1(X)-m_0(X)-\tau
+
\frac{W}{e(X)}\{Y-m_1(X)\}
-
\frac{1-W}{1-e(X)}\{Y-m_0(X)\}.
\end{equation}
Its variance, which gives the semiparametric efficiency bound, is
\begin{equation}\label{eq:eff-bound}
V_{\mathrm{eff}}
:=
\E\left[
\frac{\sigma_1^2(X)}{e(X)}
+
\frac{\sigma_0^2(X)}{1-e(X)}
+
\{m_1(X)-m_0(X)-\tau\}^2
\right],
\end{equation}
where
\[
\sigma_\omega^2(x)
:=
\Var\{Y(\omega)\mid X=x\},
\qquad
\omega\in\{0,1\}.
\]

The following efficiency result is then natural in light of the covariate-balancing structure uncovered in Section~\ref{sec:implicit-weighting} and the balancing-weight argument underlying \citet[Theorem~1]{chan2016globally}.

\begin{theorem}[Semiparametric efficiency]\label{thm:efficiency}
Suppose Assumptions~\ref{ass:dgp}-\ref{ass:efficiency} hold.  Then
\begin{equation}\label{eq:AL}
\widehat\tau_{\mathrm{imp}}-\tau
=
\frac1n\sum_{i=1}^n\psi(Z_i)
+
o_\P(n^{-1/2}).
\end{equation}
Therefore, 
\[
\sqrt n
\left(
\widehat\tau_{\mathrm{imp}}-\tau
\right)
\xrightarrow{d}
N(0,V_{\mathrm{eff}}),
\]
and hence $\widehat\tau_{\mathrm{imp}}$ attains the semiparametric efficiency bound for the ATE.
\end{theorem}

The efficiency conclusion in Theorem~\ref{thm:efficiency} should also be viewed in light of \citet[Theorem 2]{wang2023distributed}; see, also, \citet[Theorem 3.1]{imbens2005mean}. Our contribution is therefore not the efficiency result in isolation, but its structural interpretation in the causal setting.  The least-squares normal equations generate exact finite-sample balancing weights whose population counterparts are weighted $L^2$ projections of the inverse propensity scores.  Together with Theorem~\ref{thm:implicit-dr}, this shows that the product-bias phenomenon discussed in
\citet{imbens2005mean} and \citet{wang2023distributed}, exact covariate balance, inverse-propensity projection, double robustness, and semiparametric efficiency all arise from the same least-squares geometry.

The proof of Theorem~\ref{thm:efficiency} also differs in organization from that of Theorem~2 in \cite{wang2023distributed}.  Their proof decomposes the centralized sieve-imputation estimator into several estimation, approximation, and empirical-process terms, and additionally controls the error from the distributed iterative algorithm.  Our proof instead starts from the implicit weighting representation.  For each treatment arm, define
\[
g_{\omega,K}(Z)
=
m_{\omega,K}(X)
+
\delta_\omega w_{\omega,K}(X)
\{Y-m_{\omega,K}(X)\}.
\]
We show directly that
$\widehat\mu_\omega=\PP_n g_{\omega,K}+o_\P(n^{-1/2})$; its population bias is controlled by Theorem~\ref{thm:implicit-dr}, while $g_{\omega,K}$ converges in $L^2$ to the corresponding efficient influence-function component as $m_{\omega,K}$ and $w_{\omega,K}$ approximate $m_\omega$ and $1/\pi_\omega$.  Thus, although the key projection insight is shared with \cite{wang2023distributed},  the present proof is organized directly around the implicit weighting structure in Propositions~\ref{prop:implicit-weighting} and~\ref{prop:inverse-ps-projection}.

\section{Discussion}\label{sec:discussion}

Least-squares sieve regression is a classical tool of nonparametric estimation, and its limitations are perhaps equally classical.  For generic smooth functions of many covariates, the dimension of the approximating space must grow rapidly with the covariate dimension, leading to the familiar curse of dimensionality; see, e.g., \citet[Section~4.5]{wasserman2006all}.  This limitation has contributed to the increasing use of more adaptive regression architectures, most notably deep neural networks, whose approximation properties can exploit structural features such as sparsity, compositionality, and low-dimensional representations \citep{schmidt2020nonparametric}.

We conclude this paper by offering a broader perspective.  In particular, when trained under squared-error loss, neural-network regression remains, at its core, a least-squares regression procedure.  The same observation applies to transformer-based architectures when they are used as regression learners with a squared-error objective.  Thus, the distinction between classical series regression and modern machine learning is not necessarily one between least squares and something fundamentally different, but often one between a fixed, explicitly specified feature space and a rich feature representation learned from data.

This connection becomes especially transparent in the neural tangent kernel (NTK) regime.  For sufficiently wide neural networks, gradient-based training can, after linearization around initialization, be described by kernel gradient descent with the NTK; under squared loss, the resulting function-space dynamics correspond to kernel least-squares regression \citep{jacot2018neural}.  Related kernel limits have also been developed for attention-based architectures \citep{hron2020infinite}.  Kernel regression, in turn, is closely connected to sieve regression.  If a positive-definite kernel admits the expansion
\[
\mathcal K(x,x')
=
\sum_{j\geq1}\lambda_j\phi_j(x)\phi_j(x'),
\]
then truncating this expansion yields the finite-dimensional feature map
\[
V_K(x)
=
\bigl(
\sqrt{\lambda_1}\phi_1(x),\ldots,
\sqrt{\lambda_K}\phi_K(x)
\bigr)^\top,
\]
so that least-squares regression on these features is precisely a sieve regression.  Alternatively, random-feature approximations replace the kernel by a finite random expansion and again reduce the problem, under squared loss, to linear regression on a growing collection of features \citep{rahimi2007random}.  In this sense, classical sieves and modern kernel or wide-network methods can be viewed as different realizations of the same basic principle: construct a rich feature space and perform least-squares projection within that space.

This perspective gives a broader interpretation to the results of this paper.  Our analysis shows that least-squares fitting does more than estimate an outcome regression.  Its normal equations simultaneously generate balancing weights whose population counterparts approximate inverse propensity scores, while double robustness and efficiency emerge from the interaction of these two projections.  The message is therefore not limited to polynomial, spline, or other classical sieve bases.  Rather, it suggests that sufficiently rich modern least-squares regression methods may themselves contain much of the weighting structure that is typically introduced separately in causal inference.

A further practical appeal of the imputation perspective is that it provides a particularly direct interface with modern machine-learning methods.  Outcome imputation is itself a standard regression problem, so flexible regression architectures can be incorporated simply by replacing the outcome-regression learner.  By comparison, incorporating similarly rich function classes into balancing-weight procedures typically requires specifying or solving an additional balancing, dual, or representer-estimation problem.  This distinction does not preclude the use of modern machine learning for balancing weights, but it makes imputation a particularly natural plug-in point for off-the-shelf regression methods.

This interpretation should nevertheless be made with some care.  The exact results established in this paper rely on a linear sieve and its exact least-squares normal equations.  A finite-width neural network or a general transformer need not admit the same representation outside regimes, such as the NTK regime, in which an approximately fixed feature structure emerges.  Regularization, early stopping, and imperfect numerical optimization may also perturb the relevant normal equations.  Subject to these qualifications, the paper supports a simple view of causal estimation: \emph{imputation, together with a flexible outcome regression fitted by least squares, may be all you need.}  Understanding when this principle extends from classical sieves to deep neural networks and transformer-based regression would, in our view, provide a natural direction for future research.

\section{Proofs}\label{sec:proofs}

Throughout the proofs, for any two sequences $a_n$ and $b_n$, we write
\[
a_n\lesssim b_n
\]
if there exists a universal constant $C<\infty$ such that $|a_n|\le C|b_n|$ for all sufficiently large $n$. Let $e_1=(1,0,\ldots,0)^\top\in\mathbb R^K$.  Because $v_{1K}\equiv1$, $q_K$ in \eqref{eq:Q-q-w} satisfies
\begin{equation}\label{eq:q-Q-e1}
q_K
=
\E\{V_K(X)\}
=
Q_Ke_1,
\end{equation}
and hence Assumption~\ref{ass:sieve}(ii) implies $\|q_K\|=O(1)$.

In the following, we will repeatedly use the following standard concentration consequence of Assumption~\ref{ass:sieve}.  For each $\omega\in\{0,1\}$,
\begin{equation}\label{eq:gram-concentration}
\|\widehat Q_{\omega,K}-Q_{\omega,K}\|_{\op}
=
O_\P\left(
\zeta_K\sqrt{\frac{\log (2K)}{n}}
+
\frac{\zeta_K^2\log (2K)}{n}
\right)
=
o_\P(1),
\end{equation}
where $\|\cdot\|_{\op}$ stands for the matrix spectral norm. Equation \eqref{eq:gram-concentration} is from \citet[Lemma 2.2]{chen2018optimal}; see, also, \citet[Theorem 4.6]{belloni2015some} and \citet[Proof of Theorem 4.1]{cattaneo2025rosenbaum}. To see \eqref{eq:gram-concentration}, we observe $\|\delta_\omega V_KV_K^\top\|_{\op}\le\zeta_K^2$, while
\[
(\delta_\omega V_KV_K^\top)^2
=
\delta_\omega\|V_K\|^2V_KV_K^\top
\preceq
\zeta_K^2\delta_\omega V_KV_K^\top,
\]
so \eqref{eq:gram-concentration} follows from matrix Bernstein's inequality and the uniform eigenvalue bound for $Q_{\omega,K}$.  Equation \eqref{eq:gram-concentration} then implies that $\widehat Q_{\omega,K}$ is nonsingular with probability tending to one and
\begin{equation}\label{eq:inverse-bound}
\|\widehat Q_{\omega,K}^{-1}\|_{\op}=O_\P(1).
\end{equation}

\subsection{Proof of Proposition~\ref{prop:implicit-weighting}}

\begin{proof}
Because $v_{1K}\equiv1$, the intercept component of the least-squares normal equations associated with \eqref{eq:beta-hat} gives
\[
\PP_n\left[
\delta_\omega\{Y-\widehat m_\omega(X)\}
\right]
=0.
\]
Therefore,
\begin{align*}
\widehat\mu_\omega
&=
\PP_n\left[
\delta_\omega Y+(1-\delta_\omega)\widehat m_\omega(X)
\right]=
\PP_n\widehat m_\omega(X)
=
\overline V_K^\top\widehat\beta_\omega.
\end{align*}
The full normal equations yield
\[
\widehat Q_{\omega,K}\widehat\beta_\omega
=
\PP_n\{\delta_\omega V_K(X)Y\}.
\]
Whenever $\widehat Q_{\omega,K}$ is nonsingular, we then have
\begin{align*}
\widehat\mu_\omega
&=
\overline V_K^\top
\widehat Q_{\omega,K}^{-1}
\PP_n\{\delta_\omega V_K(X)Y\}=
\PP_n\{\delta_\omega\widehat w_{\omega,K}(X)Y\},
\end{align*}

Moreover,
\begin{align*}
\PP_n\{\delta_\omega\widehat w_{\omega,K}(X)V_K(X)\}
&=
\PP_n\{\delta_\omega V_K(X)V_K(X)^\top\}
\widehat Q_{\omega,K}^{-1}\overline V_K=
\overline V_K,
\end{align*}

This  completes the proof.
\end{proof}

\subsection{Proof of Proposition \ref{prop:population-balance}}
\begin{proof}
Fix $\omega\in\{0,1\}$.  We have
\begin{align*}
\E\!\left\{
\delta_\omega
w_{\omega,K}(X)V_K(X)
\right\}
&=
\E\!\left[
\delta_\omega
V_K(X)V_K(X)^\top
\right]
Q_{\omega,K}^{-1}q_K
=
Q_{\omega,K}Q_{\omega,K}^{-1}q_K
=
q_K
=
\E\{V_K(X)\}.
\end{align*}
Now let $a\in\mathcal V_K$.  Then there exists $c\in\mathbb R^K$ such that $a(x)=c^\top V_K(x)$. Hence, 
\begin{align*}
\E\!\left\{
\delta_\omega
w_{\omega,K}(X)a(X)
\right\}=
c^\top
\E\!\left\{
\delta_\omega
w_{\omega,K}(X)V_K(X)
\right\}
=
c^\top\E\{V_K(X)\}
=
\E\{a(X)\},
\end{align*}
which proves \eqref{eq:population-balance-function}.
Taking $\omega=1$ and $\omega=0$ gives the final three-way
balancing identity.
\end{proof}

\subsection{Proof of Proposition~\ref{prop:inverse-ps-projection}}

\begin{proof}
Consider a generic element $V_K^\top b\in\mathcal V_K$.  The weighted $L^2$ projection of $1/\pi_\omega$ onto $\mathcal V_K$ solves
\[
\min_{b\in\mathbb R^K}
\E\left[
\delta_\omega
\left\{
\frac1{\pi_\omega(X)}-V_K(X)^\top b
\right\}^2
\right].
\]
The corresponding normal equations are
\[
\E\left[
\delta_\omega V_K(X)
\left\{
\frac1{\pi_\omega(X)}-V_K(X)^\top b
\right\}
\right]
=0.
\]
Since
\[
\E\left\{
\frac{\delta_\omega}{\pi_\omega(X)}V_K(X)
\right\}
=
\E\left[
\E\left\{
\frac{\delta_\omega}{\pi_\omega(X)}V_K(X)
\mid X
\right\}
\right]
=
\E\{V_K(X)\}
=
q_K,
\]
the normal equations reduce to
\[
q_K-Q_{\omega,K}b=0.
\]
Nonsingularity of $Q_{\omega,K}$ therefore gives the unique solution
\[
b=Q_{\omega,K}^{-1}q_K,
\]
and hence the projection is $w_{\omega,K}$, proving \eqref{eq:inverse-ps-projection}.

Finally,
\[
\rho_\omega=\frac1{\pi_\omega}-1.
\]
Since $1\in\mathcal V_K$, the constant function is projected onto itself.  By linearity of orthogonal projection,
\[
\Pi_{\mathcal V_K}^{(\omega)}\rho_\omega
=
\Pi_{\mathcal V_K}^{(\omega)}
\left(\frac1{\pi_\omega}-1\right)
=
w_{\omega,K}-1,
\]
which proves the final assertion.
\end{proof}

\subsection{Proof of Theorem~\ref{thm:implicit-dr}}

\begin{proof}
Fix $\omega\in\{0,1\}$.  By unconfoundedness,
\[
\E(Y\mid X,W=\omega)=m_\omega(X).
\]
The population least-squares normal equations associated with \eqref{eq:beta-star} therefore imply
\begin{equation}\label{eq:population-normal}
\E\left[
\delta_\omega V_K(X)
\{m_\omega(X)-m_{\omega,K}(X)\}
\right]
=0.
\end{equation}
In particular, because the sieve contains an intercept,
\begin{equation}\label{eq:population-intercept}
\E\left[
\delta_\omega
\{m_\omega(X)-m_{\omega,K}(X)\}
\right]
=0.
\end{equation}

Let
\[
r_{\omega,K}(X)
:=
m_{\omega,K}(X)-m_\omega(X).
\]
The contribution of treatment arm $\omega$ to the bias of the population imputation functional is
\[
\E\{(1-\pi_\omega(X))r_{\omega,K}(X)\}.
\]
Using \eqref{eq:population-intercept},
\[
\E\{(1-\pi_\omega)r_{\omega,K}\}
=
\E r_{\omega,K}
-
\E\{\pi_\omega r_{\omega,K}\}
=
\E r_{\omega,K}
=
B_{\omega,K}.
\]
Taking the treated contribution minus the control contribution gives \eqref{eq:dr-bias-compact}.

Next,
\[
B_{\omega,K}
=
\E r_{\omega,K}(X)
=
\E\left\{
\frac{\delta_\omega}{\pi_\omega(X)}
r_{\omega,K}(X)
\right\}.
\]
Because $w_{\omega,K}\in\mathcal V_K$, \eqref{eq:population-normal} gives
\[
\E\{\delta_\omega w_{\omega,K}(X)r_{\omega,K}(X)\}=0.
\]
Subtracting this zero term yields
\[
B_{\omega,K}
=
\E\left[
\delta_\omega
\left\{
\frac1{\pi_\omega(X)}-w_{\omega,K}(X)
\right\}
r_{\omega,K}(X)
\right],
\]
which is \eqref{eq:dr-bias-arm}.  Cauchy-Schwarz and Equations \eqref{eq:Delta-m} and \eqref{eq:Delta-rho-w} imply
\[
|B_{\omega,K}|
\le
\Delta_{\rho,\omega,K}\Delta_{m,\omega,K}.
\]
Summing the two treatment-arm contributions proves \eqref{eq:product-bound}.
\end{proof}

\subsection{Proof of Corollary~\ref{cor:implicit-dr}}

\begin{proof}
We first show that the sample estimator is consistent for the population imputation functional.  Fix $\omega\in\{0,1\}$ and define
\[
S_{\omega,K}
:=
\PP_n\left[
\delta_\omega V_K(X)
\{Y-m_{\omega,K}(X)\}
\right].
\]
The population normal equations imply $\E(S_{\omega,K})=0$.  Moreover, since $m_{\omega,K}$ is the weighted $L^2$ projection of $m_\omega$,
\[
\E\{\delta_\omega(m_\omega-m_{\omega,K})^2\}
\le
\E\{\delta_\omega m_\omega^2\}
\lesssim 1,
\]
and Assumption~\ref{ass:moments} also gives
\[
\E\{\delta_\omega(Y-m_{\omega,K})^2\}\lesssim 1.
\]
Consequently,
\[
\E\|S_{\omega,K}\|^2
\le
\frac{\zeta_K^2}{n}
\E\{\delta_\omega(Y-m_{\omega,K})^2\}
\lesssim
\frac{\zeta_K^2}{n},
\]
so that
\begin{equation}\label{eq:S-consistency-rate}
\|S_{\omega,K}\|
=
O_\P\left(\frac{\zeta_K}{\sqrt n}\right)
=
o_\P(1).
\end{equation}

From the least-squares normal equations,
\[
\widehat\beta_\omega-\beta_{\omega,K}^*
=
\widehat Q_{\omega,K}^{-1}S_{\omega,K},
\]
and therefore
\begin{align}
\widehat\mu_\omega
&=
\PP_n m_{\omega,K}(X)
+
\overline V_K^\top
\widehat Q_{\omega,K}^{-1}
S_{\omega,K}.
\label{eq:mu-consistency-decomp}
\end{align}
By \eqref{eq:q-Q-e1},
\[
\|\overline V_K\|
\le
\|q_K\|+\|\overline V_K-q_K\|
=
O_\P(1),
\]
because
\[
\E\|\overline V_K-q_K\|^2
=
\frac1n
\E\|V_K(X)-q_K\|^2
\lesssim
\frac K n
=o(1).
\]
Combining this with \eqref{eq:inverse-bound} and \eqref{eq:S-consistency-rate} shows that the second term in \eqref{eq:mu-consistency-decomp} is $o_\P(1)$.

The projection property and overlap imply
\[
\E\{m_{\omega,K}(X)^2\}
\le
c_e^{-1}\E\{\delta_\omega m_{\omega,K}(X)^2\}
\le
c_e^{-1}\E\{\delta_\omega m_\omega(X)^2\}
\lesssim 1.
\]
Moreover,
\[
\E\left[\{(\PP_n-\P)m_{\omega,K}\}^2\right]
=
\frac1n\Var\{m_{\omega,K}(X)\}
\le
\frac1n\E\{m_{\omega,K}(X)^2\}
\lesssim
\frac1n,
\]
and hence, by Chebyshev's inequality,
\[
(\PP_n-\P)m_{\omega,K}=o_\P(1).
\]
Finally, \eqref{eq:population-intercept} implies
\[
\E\left[
\delta_\omega Y+(1-\delta_\omega)m_{\omega,K}(X)
\right]
=
\E m_{\omega,K}(X).
\]
Thus
\[
\widehat\mu_\omega
-
\E\left[
\delta_\omega Y+(1-\delta_\omega)m_{\omega,K}(X)
\right]
=
o_\P(1).
\]
Taking the difference between $\omega=1$ and $\omega=0$ proves \eqref{eq:sample-pop}.

If the product of approximation errors tends to zero, Theorem~\ref{thm:implicit-dr} gives $\tau_K-\tau=o(1)$, and therefore
\[
\widehat\tau_{\mathrm{imp}}-\tau
=
(\widehat\tau_{\mathrm{imp}}-\tau_K)
+
(\tau_K-\tau)
=
o_\P(1).
\]
This completes the proof.
\end{proof}

\subsection{Proof of Theorem~\ref{thm:efficiency}}

\begin{proof}
We derive an asymptotic expansion for each potential-outcome mean.  Fix $\omega\in\{0,1\}$ and continue to write
\[
S_{\omega,K}
=
\PP_n\left[
\delta_\omega V_K(X)
\{Y-m_{\omega,K}(X)\}
\right].
\]
As in \eqref{eq:mu-consistency-decomp},
\[
\widehat\mu_\omega
=
\PP_n m_{\omega,K}(X)
+
\overline V_K^\top
\widehat Q_{\omega,K}^{-1}
S_{\omega,K}.
\]
Define
\[
g_{\omega,K}(Z)
:=
m_{\omega,K}(X)
+
\delta_\omega w_{\omega,K}(X)
\{Y-m_{\omega,K}(X)\}.
\]
We first show that
\begin{equation}\label{eq:series-linearization}
\widehat\mu_\omega
=
\PP_n g_{\omega,K}
+
o_\P(n^{-1/2}).
\end{equation}

Let
\[
D_{\omega,K}^\top
:=
\overline V_K^\top\widehat Q_{\omega,K}^{-1}
-
q_K^\top Q_{\omega,K}^{-1}.
\]
Then the remainder in \eqref{eq:series-linearization} is $D_{\omega,K}^\top S_{\omega,K}$.  First, we claim
\[
\|\overline V_K-q_K\|
=
O_\P\left(\sqrt{\frac K n}\right).
\]
To prove it, we have
\[
\overline V_K-q_K
=
\frac1n\sum_{i=1}^n
\bigl\{V_K(X_i)-q_K\bigr\}.
\]
Since $\E\{V_K(X_i)-q_K\}=0$, we obtain
\begin{align*}
\E\|\overline V_K-q_K\|^2
&=
\frac{1}{n^2}
\E\left\|
\sum_{i=1}^n
\{V_K(X_i)-q_K\}
\right\|^2
=
\frac1n
\E\|V_K(X)-q_K\|^2
\\
&\le
\frac1n
\E\|V_K(X)\|^2
=
\frac1n
\operatorname{tr}
\E\{V_K(X)V_K(X)^\top\}
=
\frac1n
\operatorname{tr}(Q_K).
\end{align*}
Under the eigenvalue condition that $\lambda_{\max}(Q_K)\lesssim 1$,  we then have
\[
\operatorname{tr}(Q_K)
\le
K\,\lambda_{\max}(Q_K)
\le
CK
\]
so that $\E\|\overline V_K-q_K\|^2\lesssim
\frac Kn$. Markov's inequality then yields the claim.

For the inverse-matrix contribution, write
\begin{align*}
q_K^\top(\widehat Q_{\omega,K}^{-1}-Q_{\omega,K}^{-1})
&=
q_K^\top Q_{\omega,K}^{-1}
(Q_{\omega,K}-\widehat Q_{\omega,K})
\widehat Q_{\omega,K}^{-1}.
\end{align*}
Set
\[
\gamma_{\omega,K}:=Q_{\omega,K}^{-1}q_K,
\qquad
w_{\omega,K}(x)=V_K(x)^\top\gamma_{\omega,K}.
\]
By Assumption~\ref{ass:efficiency}(iii),
\[
\left\|
\gamma_{\omega,K}^\top
(\widehat Q_{\omega,K}-Q_{\omega,K})
\right\|
=
\left\|
(\PP_n-\P)
\{\delta_\omega w_{\omega,K}(X)V_K(X)^\top\}
\right\|
=
O_\P\left(\sqrt{\frac K n}\right),
\]
because
\[
\E\{\delta_\omega w_{\omega,K}(X)^2\|V_K(X)\|^2\}
\le
C_w^2\,\operatorname{tr}(Q_{\omega,K})
\lesssim K.
\]
Together with \eqref{eq:inverse-bound}, this yields
\begin{equation}\label{eq:D-rate}
\|D_{\omega,K}\|
=
O_\P\left(\sqrt{\frac K n}\right).
\end{equation}

Next decompose
\[
S_{\omega,K}
=
\PP_n\{\delta_\omega V_K(X)(Y-m_\omega(X))\}
+
\PP_n\{\delta_\omega V_K(X)(m_\omega(X)-m_{\omega,K}(X))\}.
\]
Both summands have mean zero.  Assumption~\ref{ass:moments}, overlap, and the uniform eigenvalue bounds imply
\[
\left\|
\PP_n\{\delta_\omega V_K(X)(Y-m_\omega(X))\}
\right\|
=
O_\P\left(\sqrt{\frac K n}\right),
\]
whereas the second term satisfies
\[
\left\|
\PP_n\{\delta_\omega V_K(X)(m_\omega(X)-m_{\omega,K}(X))\}
\right\|
=
O_\P\left(
\frac{\zeta_K\Delta_{m,\omega,K}}{\sqrt n}
\right).
\]
Hence
\begin{equation}\label{eq:S-eff-rate}
\|S_{\omega,K}\|
=
O_\P\left(
\sqrt{\frac K n}
+
\frac{\zeta_K\Delta_{m,\omega,K}}{\sqrt n}
\right).
\end{equation}

Since
\[
\zeta_K^2
\ge
\E\|V_K(X)\|^2
=
\operatorname{tr}(Q_K)
\ge
c_QK,
\]
Assumption~\ref{ass:efficiency}(i) implies $K/\sqrt n\to0$ as well as
\[
\frac{\zeta_K\sqrt K}{\sqrt n}\to0.
\]
Combining \eqref{eq:D-rate}, \eqref{eq:S-eff-rate}, and $\Delta_{m,\omega,K}\to0$ gives
\[
\sqrt n\,
|D_{\omega,K}^\top S_{\omega,K}|
=
o_\P(1),
\]
which proves \eqref{eq:series-linearization}.

By the population normal equations \eqref{eq:population-normal} and the fact that $w_{\omega,K}\in\mathcal V_K$,
\[
\E\left[
\delta_\omega w_{\omega,K}(X)
\{m_\omega(X)-m_{\omega,K}(X)\}
\right]
=0.
\]
Therefore,
\[
\P g_{\omega,K}
=
\E m_{\omega,K}(X)
=
\mu_\omega+B_{\omega,K}.
\]
It follows from \eqref{eq:series-linearization} that
\begin{equation}\label{eq:mu-bias-expansion}
\widehat\mu_\omega-\mu_\omega
=
(\PP_n-\P)g_{\omega,K}
+
B_{\omega,K}
+
o_\P(n^{-1/2}).
\end{equation}
By Theorem~\ref{thm:implicit-dr} and Assumption~\ref{ass:efficiency}(ii),
\begin{equation}\label{eq:bias-rootn}
\sqrt n\,|B_{\omega,K}|
\le
\sqrt n\,
\Delta_{\rho,\omega,K}\Delta_{m,\omega,K}
\to0.
\end{equation}

Now define
\[
g_\omega(Z)
:=
m_\omega(X)
+
\frac{\delta_\omega}{\pi_\omega(X)}
\{Y-m_\omega(X)\}.
\]
We claim that
\begin{equation}\label{eq:g-L2}
\E\{g_{\omega,K}(Z)-g_\omega(Z)\}^2
\to0.
\end{equation}
To prove \eqref{eq:g-L2}, we observe that
\begin{align*}
g_{\omega,K}-g_\omega
={}&
\{m_{\omega,K}(X)-m_\omega(X)\}
\{1-\delta_\omega w_{\omega,K}(X)\}+
\delta_\omega
\left\{
w_{\omega,K}(X)-\frac1{\pi_\omega(X)}
\right\}
\{Y-m_\omega(X)\}.
\end{align*}
For the first term, overlap and Assumption~\ref{ass:efficiency}(iii) give
\[
\E\left[
\{m_{\omega,K}-m_\omega\}^2
\{1-\delta_\omega w_{\omega,K}\}^2
\right]
\lesssim
\E(m_{\omega,K}-m_\omega)^2
\lesssim
\Delta_{m,\omega,K}^2
\to0.
\]
For the second term, Assumption~\ref{ass:moments}, unconfoundedness, and \eqref{eq:Delta-rho-w} give
\[
\E\left[
\delta_\omega
\left\{
w_{\omega,K}-\frac1{\pi_\omega}
\right\}^2
\{Y-m_\omega\}^2
\right]
\lesssim
\Delta_{\rho,\omega,K}^2
\to0.
\]
This proves \eqref{eq:g-L2}.  Since $g_{\omega,K}$ is nonrandom for each $n$,
\[
\sqrt n(\PP_n-\P)(g_{\omega,K}-g_\omega)
=
o_\P(1)
\]
by Chebyshev's inequality.  Combining this fact with \eqref{eq:mu-bias-expansion} and \eqref{eq:bias-rootn} yields
\begin{equation}\label{eq:mu-AL}
\widehat\mu_\omega-\mu_\omega
=
\frac1n\sum_{i=1}^n
\{g_\omega(Z_i)-\mu_\omega\}
+
o_\P(n^{-1/2}).
\end{equation}

Taking the difference between $\omega=1$ and $\omega=0$ in \eqref{eq:mu-AL} gives
\[
\widehat\tau_{\mathrm{imp}}-\tau
=
\frac1n\sum_{i=1}^n\psi(Z_i)
+
o_\P(n^{-1/2}),
\]
with $\psi$ defined in \eqref{eq:eif}.  This proves \eqref{eq:AL}.

Under overlap and Assumption~\ref{ass:moments}, $\psi(Z)$ has finite variance.  The classical central limit theorem therefore gives
\[
\sqrt n(\widehat\tau_{\mathrm{imp}}-\tau)
\xrightarrow{d}
N(0,\Var\{\psi(Z)\}).
\]
A direct conditional-variance calculation gives
\[
\Var\{\psi(Z)\}
=
V_{\mathrm{eff}},
\]
where $V_{\mathrm{eff}}$ is defined in \eqref{eq:eff-bound}.  
\end{proof}

\section*{Acknowledgment of AI use}

The authors used AI during the preparation of this manuscript for language editing, literature exploration and cross-checking, and assistance in reviewing the exposition and technical arguments.  All mathematical results, proofs, interpretations, and cited references were independently checked and verified by the authors, who take full responsibility for the contents of the paper.

{
\bibliographystyle{apalike}
\bibliography{AMS}
}

\end{document}